\documentclass[11pt]{article}

\usepackage[margin=1in]{geometry}

\usepackage{lmodern}
\usepackage[T1]{fontenc}

\usepackage[english]{babel} 
\usepackage[utf8]{inputenc}

\usepackage{mathrsfs} 
\usepackage{amsfonts}
\usepackage{marvosym} 
\usepackage{mathtools,amssymb,amsthm}

\usepackage{bm}

\usepackage[dvipsnames, svgnames]{xcolor}
\definecolor{vert}{RGB}{13,75,47}
\definecolor{gris}{RGB}{128,128,128}
\definecolor{bleu}{RGB}{0,50,150}
\definecolor{rouge}{RGB}{162,30,2}
\usepackage[colorlinks = true,linkcolor=vert,citecolor=rouge]{hyperref}

\usepackage[nameinlink]{cleveref} 
\crefname{equation}{}{}
\usepackage{bbm}

\usepackage{epsfig}

\usepackage{tikz-cd} 

\usepackage{stmaryrd}
\usepackage{manfnt}
\usepackage{multicol}
\usepackage{array}
\usepackage{multirow}

\usepackage{fancyhdr} 
\usepackage{fancybox} 

\usepackage{calligra}
\title{The weight filtration on real singular homology is motivic}
\author{Raphaël Ruimy}
\date{}

\usepackage{tikz}
\usetikzlibrary{calc, intersections}

\usepackage{adjustbox}

\theoremstyle{plain}
\newtheorem{theorem}{Theorem}
\newtheorem{proposition}[theorem]{Proposition}

\newtheorem*{theorem*}{Theorem}

\newtheorem*{prop*}{Proposition}

\newtheorem{lemma}[theorem]{Lemma}
\newtheorem{corollary}[theorem]{Corollary} 

\theoremstyle{definition}
\newtheorem{definition}[theorem]{Definition}
\newtheorem*{definition*}{Definition}
\crefname{hyp}{Hypothesis}{Hypotheses}

\newtheorem{constr}[theorem]{Construction}

\theoremstyle{remark}

\newtheorem{rem}[theorem]{Remark}
\newtheorem*{rem*}{Remark}
 
\newtheorem{ex}[theorem]{Example}

\numberwithin{theorem}{section}
\numberwithin{equation}{section}

\newcommand{\C}{\mathbb{C}}
\newcommand{\R}{\mathbb{R}}
\newcommand{\Q}{\mathbb{Q}}
\newcommand{\Z}{\mathbb{Z}}

\newcommand{\Spec}{\mathrm{Spec}}

\newcommand{\HH}{\mathrm{H}}
\newcommand{\D}{\mathrm{D}}

\newcommand{\ocal}{\mathcal{O}}

\newcommand{\ecal}{\mathcal{E}}

\newcommand{\scal}{\mathcal{S}}

\newcommand{\ccal}{\mathcal{C}}

\newcommand{\dcal}{\mathcal{D}}

\newcommand{\mc}{\mathcal}

\newcommand{\mb}{\mathbb}

\newcommand{\Ind}{\mathrm{Ind}}

\newcommand{\colim}{\mathrm{colim}}

\newcommand{\Hom}{\mathrm{Hom}}

\newcommand{\F}{\mathbb{F}}

\newcommand{\hcal}{\mathcal{H}}

\DeclareMathOperator{\sHom}{\mathscr{H}\text{\kern -3pt {\calligra\large om}}\,}

\newcommand{\Map}{\mathrm{Map}}
\newcommand{\map}{\mathrm{map}}
\newcommand{\Fun}{\mathrm{Fun}}
\newcommand{\CAlg}{\mathrm{CAlg}}

\newcommand{\DM}{\mathrm{DM}}

\newcommand{\Sh}{\mathrm{Sh}}

\newcommand{\Sch}{\mathrm{Sch}}
\newcommand{\op}{\mathrm{op}}

\newcommand{\catinfty}{\mathrm{Cat}_\infty}

\newcommand{\PSh}{\operatorname{\mathrm{PSh}}}

\newcommand{\Sm}{\mathrm{Sm}}

\newcommand{\AAA}{\mathbb{A}^1}
\newcommand{\eff}{\mathrm{eff}}

\newcommand{\rmm}{\mathrm{M}}

\usepackage{appendix}
\makeatother

\begin{document}

\maketitle

\begin{abstract}
    We give an alternative construction of Totaro's weight filtration on singular homology of the real points of a real algebraic variety. Our construction shows that this filtration comes from Bondarko's weight filtration on Voevodsky motives. 
\end{abstract}
\tableofcontents
\section*{Intoduction}
Deligne introduced in \cite{hodgeII,hodgeIII} a canonical filtration on the singular cohomology groups with $\Q$-coefficients of a complex algebraic variety called the \emph{weight filtration}. He also introduced a weight filtration on the $\ell$-adic cohomology groups in \cite{WeilII}. The existence of both filtrations was initially envisioned by Grothendieck as a shadow of a motivic weight filtration. This vision was later justified by several authors. First, Voevodsky constructed in \cite[Chapter~5]{orange} the category $\DM_\mathrm{gm}(k,\Lambda)$ for $k$ a perfect field and $\Lambda$ any commutative ring (see \Cref{constrDMgm}). Assume that the characteristic exponent $p$ of $k$ is invertible in $\Lambda$. In \cite{thesekelly}, Kelly constructed\footnote{following Voevodsky's construction is the case when $k$ allows resolutions of singularities.} a functor 
\[M\colon \Sch_k\to \DM_\mathrm{gm}(k,\Lambda)\] from the category $\Sch_k$ of $k$-schemes of finite type. The objects of $\DM_\mathrm{gm}(k,\Lambda)$ are furthermore equipped with a weight filtration using Bondarko's work \cite{bondarkoZ[1/p]-motivic}. It is then possible to show that Ayoub's Betti realization 
\[\rho_\C\colon \DM_\mathrm{gm}(\C,\Q)\to \D^b(\Q)\] from \cite{ayobetti} sends the motive $M(X)$ to the singular chain complex $C_*(X(\C),\Q)$ and that the weight filtration on $M(X)$ is sent to the weight filtration on singular homology. A similar result is true for the étale and $\ell$-adic realizations from \cite{ayo14,em} and Deligne's weight filtration on $\ell$-adic complexes.

Our goal is to tell the same story in the real world. The existence of a weight filtration on the singular cohomology of the real points of a real algebraic variety with coefficients in the field $\F_2$ with two elements was stated by Totaro in \cite{totaro} and was later proved in \cite{mccrory1} by McCrory and Parusi{\'n}ski. In the case of Borel-Moore homology, their construction consists of formally extending the trivial filtration on projective smooth $\R$-schemes using an extension result of Guillén and Navarro Aznar \cite{Guillen-NA}; their construction for ordinary homology in \cite{mccrory2} is more involved as it requires constructing Gysin maps by hand. 
The goal of this note is to give an alternative construction of this weight filtration by using the weight filtration on motives. This will also show that the weight filtration is a realization of the motivic weight filtration. 

The first step is to construct the real motivic realization functor \[\rho_\R\colon \DM_\mathrm{gm}(\R,\F_2)\to \D^b(\F_2).\] As singular cohomology of the real points is not an étale-local invariant, the construction is more involved than the constructions in the complex case or for $\ell$-adic cohomology: in the étale world, every sheaf has canonical transfers which is not the case for Nisnevich sheaves. We must first adapt Robalo's universal property of motivic sheaves from \cite{rob2} to the setting of Voevodsky motives. This requires some $\infty$-categorical constructions which are done in \Cref{inftycatcons}. The universal property is then \Cref{universalprtyDM}. Constructing the realization then amounts to constructing a symmetric monoidal additive functor 
\[\Sm^\mathrm{cor}(\R,\F_2)\to \D^b(\F_2)\]
where $\Sm^\mathrm{cor}(\R,\F_2)$ is the category of finite correspondences over smooth $\R$-schemes with $\F_2$-coefficients (see \Cref{defcorr}). This functor must factor the usual singular chain complex functor $C_*(-(\R),\F_2)$. There is a fairly natural way to assign to an object (resp. a map) in $\Sm^\mathrm{cor}(\R,\F_2)$ an object (resp. a map) in $ \D^b(\F_2)$ (see \Cref{concretedescription}) but with this naive description, it is very hard to tell why the assignment is indeed a functor as checking compatibility with composition seems very hard. To avoid this issue, we use Jacobson's theorem (\Cref{jacobson}) which identifies real cohomology with Zariski cohomology of a certain sheaf $\overline{\mc{I}}^\infty:=\colim_m  \ \overline{\mc{I}}^m$ which is defined from the powers of the fundamental ideal of the Witt ring (see \Cref{constrjacmap}). We then use the existence of a Gersten resolution for $\overline{\mc{I}}^m$ to conclude. This latter point is in fact very subtle as it relies on the Milnor conjecture for K-theory which was proved in \cite{ovv}. On Chow motives, our construction recovers that of Fu in \cite{furealreal} (see \Cref{furealismyreal}).

Using Bondarko's theory of weight structures and its interpretation by Sosnilo (and Aoki in the monoidal context) in \cite{sosnilonegativeKtheory,aoki} (see \Cref{aokiji}), we then show that $\rho_\R$ factors through the filtered bounded derived category $\D^b\mathrm{F}(\F_2)$, more explicitly, we construct a \emph{filtered real realization functor} $\rho_\R^\mc{W}\colon \mathrm{DM}_\mathrm{gm}(\R,\F_2)\to\D^b\mathrm{F}(\F_2)$ such that the composition 
\[\mathrm{DM}_\mathrm{gm}(\R,\F_2)\xrightarrow{\rho_\R^\mc{W}}\D^b\mathrm{F}(\F_2)\xrightarrow{\varphi}\D^b(\F_2),\] with $\varphi$ the forgetful functor, is $\rho_\R$. On Chow motives, the filtration we obtain is the canonical filtration. 
To get the weight filtration on the singular homology of the real points of a real algebraic variety $X$, we simply apply $\rho_\R^\mc{W}$ to $M(X)$; which yields a filtered complex and therefore a spectral sequence whose abutment gives the weight filtration. By \Cref{mainthm}, this is an alternative construction of the weight spectral sequence and of the weight filtration on real singular homology of  \cite{mccrory1,mccrory2} and shows that is of motivic nature. We also give similar constructions for Borel-Moore homology, cohomology and cohomology with compact support.



\subsection*{Acknowledgments}
I would first like to thank Fabien Priziac for the very neat introduction to the weight filtration on real singular homology that he delivered at the workshop \emph{Real Motivic Geometry} in Le Croisic; the question of whether this weight filtration came from Bondarko's seemed very natural after hearing his presentation. I thank Erwan Brugallé, Jean-Baptiste Campesato, Adrien Dubouloz and Penka Georgieva for organizing this very nice workshop. 

I would also like to thank Frédéric Déglise for finding a gap in the definition of the real realization regarding composition of correspondences and to thank him and Jean Fasel for ideas on how to fix it. I thank Martin Gallauer for pointing out a very embarrassing $\infty$-categorical mistake. I thank Swann Tubach for proofreading the draft of this note and finding some mistakes and improvements. Finally, I would like to thank Adrien Dubouloz, Niels Feld, Lie Fu, Goulwenn Fichou, Samuel Lerbet and Massimo Pippi for their interst in this work and some great discussions. I also thank the ANR-21-CE40-0015 HQDIAG for supporting this project, as well as the University of Milan and the University of Grenoble where this research was conducted.

\section{Higher categorical constructions}\label{inftycatcons}

Let $\catinfty$ be the (large) $\infty$-category of small $\infty$-categories and $\rmm\catinfty\coloneqq\CAlg(\catinfty^\times)$ the $\infty$-category of small symmetric monoidal $\infty$-categories. 
\begin{definition}
    Let $\mathscr{K}$ be a small collection of simplicial sets. We let $\rmm\catinfty^\mathscr{K}$ be the subcategory of $\rmm\catinfty$ spanned by 
    \begin{enumerate}
        \item those small symmetric monoidal $\infty$-categories which admit $\mathscr{K}$-indexed colimits and whose tensor product preserves $\mathscr{K}$-indexed colimits separately in each variable,
        \item those symmetric monoidal $\infty$-functors which preserve $\mathscr{K}$-indexed colimits.
    \end{enumerate}
\end{definition}

\begin{rem}
    If $\mathscr{K}$ is empty, we recover $\rmm\catinfty$.
\end{rem}

We will denote by $\Sigma$ (resp. $\mathrm{rex}$\footnote{$\Sigma$ stands for "sum" and $\mathrm{rex}$ stands for "right exact".}) the collection of finite discrete (resp. finite) simplicial sets so that $\Sigma$-indexed (resp. $\mathrm{rex}$-indexed) colimits are finite sums (resp. finite colimits).
We also denote by $\mathrm{MAdd}\subseteq\rmm\catinfty^\Sigma$ the category of small additive symmetric monoidal (ordinary) categories. 
\begin{constr}
    Let $\mathrm{A}^\otimes\in \mathrm{MAdd}$, the category $\mathrm{C}^b_{\geqslant 0}(\mathrm{A})$ of bounded non-negative chain complex has a natural structure of monoidal dg-category $\bm{\mathrm{C}^b_{\geqslant 0}(\mathrm{A})^\otimes}$, the complex of maps between two chain complexes in $\mathrm{A}$ being given by 
\[\Map(C_*,D_*)_p=\prod_{n\in \Z} \Hom(C_n, D_{n+p})\]
with differentials $df= d\circ f - (-1)^p f\circ d$ and monoidal product given by \[C_*\otimes D_*=\bigoplus_{n\in \Z} (C_{*+n}\otimes D_{*-n})\] on objects.
\end{constr}  
\begin{constr}\label{doldkan}
    The Dold-Kan equivalence 
\[\mathrm{DK}\colon \mathrm{C}_{\geqslant 0}(\mathrm{Ab})\to \Fun(\Delta^\op,\mathrm{Ab})\]
can be promoted into a right-lax monoidal functor (see \cite[Example~1.2.3.26]{ha}). This allows to associate a symmetric monoidal simplicial category to any monoidal dg-category (see also \cite[Construction~1.3.1.13]{ha}): indeed let $\bm{\mathrm{C}^\otimes}$ be a symmetric monoidal dg-category, we let $\bm{\mathrm{C}}_\Delta^\otimes$ be the symmetric monoidal simplicial category such that
\begin{enumerate}
    \item The objects of $\bm{\mathrm{C}}_\Delta^\otimes$ are the objects of $\bm{\mathrm{C}}^\otimes$,
    \item If $M$ and $N$ are objects of $\bm{\mathrm{C}^\otimes}$, we set \[\Map_{\bm{\mathrm{C}}_\Delta^\otimes}(M,N)\coloneqq\mathrm{DK}(\tau_{\geqslant 0}(\Map_{\bm{\mathrm{C}^\otimes}}(M,N))),\]
    \item Composition is well-defined because the functor $\mathrm{DK}\circ\tau_{\geqslant 0}$ is right-lax monoidal,
    \item The tensor product on object is given by the tensor product of $\bm{\mathrm{C}^\otimes}$ and is defined on maps by using the right-lax monoidal constraints. 
\end{enumerate}
To $\bm{\mathrm{C}}^\otimes_\Delta$, we can associate a simplicial coloured operad by \cite[Variant~2.1.1.3]{ha} to which we can associate a symmetric monoidal $\infty$-category using \cite[Definition~2.1.1.23]{ha} that we denote by $\mathrm{N}^\otimes_\mathrm{dg}(\bm{\mathrm{C}^\otimes})$. This construction is functorial in $\bm{\mathrm{C}^\otimes}$.
\end{constr}
\begin{definition}
Let $\mathrm{A}^\otimes\in \mathrm{MAdd}$, we let 
\[\mathrm{K}^b_{\geqslant 0}(\mathrm{A})^\otimes=\mathrm{N}^\otimes_\mathrm{dg}(\bm{\mathrm{C}^b_{\geqslant 0}(\mathrm{A})^\otimes}).\] 
\end{definition}
\begin{rem}
    The underlying $\infty$-category of $\mathrm{K}^b_{\geqslant 0}(\mathrm{A})^\otimes$ is the differential-graded nerve of the dg-category $\bm{\mathrm{C}^b_{\geqslant 0}(\mathrm{A})}$. Hence, it is the localization $\mathrm{K}^b_{\geqslant 0}(\mathrm{A})$ of $\mathrm{C}^b_{\geqslant 0}(\mathrm{A})$ with respect to chain-homotopy equivalences by \cite[Remark~2.9]{cisinski-bunke2}.
\end{rem}
Let $\mathrm{A}^\otimes\in \mathrm{MAdd}$. We have a symmetric monoidal dg-functor
\[\bm{\mathrm{A}^\otimes} \to \bm{\mathrm{C}^b_{\geqslant 0}(\mathrm{A})^\otimes}\] sending an object to itself seen as a complex placed in degree $0$ (the dg-structure of $\bm{\mathrm{A}^\otimes}$ is given by the additive structure placed in degree $0$). By functoriality of \Cref{doldkan}, we get a symmetric monoidal $\infty$-functor 
\[\eta_{\mathrm{A}^\otimes}\colon \mathrm{A}^\otimes\to \mathrm{K}^b_{\geqslant 0}(\mathrm{A})^\otimes.\]

Recall from \cite[Remark~4.8.1.8]{ha} that if $\mathscr{K}\subseteq \mathscr{K}'$ the inclusion $\rmm\catinfty^{\mathscr{K}'}\hookrightarrow\rmm\catinfty^\mathscr{K}$ admits a left adjoint 
\[\mc{P}^{\mathscr{K}'}_\mathscr{K}\colon \rmm\catinfty^\mathscr{K}\to\rmm\catinfty^{\mathscr{K}'}\]
which sends a symmetric monoidal $\infty$-category $\ccal^\otimes$ that admits $\mathscr{K}$-indexed colimits to the subcategory of $\Fun_\mathscr{K}(\ccal^\op,\scal)^\otimes$ (the category of $\mathscr{K}$-indexed preserving functors from $\ccal^\op$ to the category $\scal$ of spaces, the monoidal structure being given by Day convolution) made of those objects generated by representables under $\mathscr{K}'$-indexed colimits.

\begin{proposition}\label{universalKb}
    The unit of the adjunction between $\mc{P}^\mathrm{rex}_\Sigma$ and the inclusion applied to $\mathrm{A}^\otimes\in \mathrm{MAdd}$ induces the functor $\eta_{\mathrm{A}^\otimes}$ up to a unique equivalence on the target. In particular, if $\ccal^\otimes$ belongs to $\rmm\catinfty^\mathrm{rex}$, the functor $\eta_{\mathrm{A}^\otimes}$ induces an equivalence \[\Fun_{\rmm\catinfty^\mathrm{rex}}(\mathrm{K}^b_{\geqslant 0}(\mathrm{A})^\otimes,\ccal^\otimes)\to \Fun_{\rmm\catinfty^\Sigma}(\mathrm{A}^\otimes,\ccal^\otimes).\]
\end{proposition}
\begin{proof}
    The tensor product on $\mathrm{K}^b_{\geqslant 0}(\mathrm{A})^\otimes$ preserves finite colimits separately on each variable, whence a unique symmetric monoidal map $\mc{P}^\mathrm{rex}_\Sigma(\mathrm{A}^\otimes)\to \mathrm{K}^b_{\geqslant 0}(\mathrm{A})^\otimes$ factorizing $\eta_{\mathrm{A}^\otimes}$. This map is an equivalence because it is an equivalence at the level of the underlying $\infty$-categories by the proof of \cite[Proposition 7.4.5]{cisinski-bunke}
\end{proof}

\begin{definition}
We let $\rmm\catinfty^\mathrm{st}$ be the $\infty$-category of \emph{small stable symmetric monoidal $\infty$-categories} defined as the full subcategory of $\rmm\catinfty^\mathrm{rex}$ spanned by stable $\infty$-categories.
\end{definition}
\begin{definition} Let $\mathscr{K}$ be a small collection of simplicial sets. Let $\ccal^\otimes\in \rmm\catinfty^\mathscr{K}$ and let $X$ be an object of $\ccal$. The \emph{$X$-stabilization of $\ccal$ (with $\mathscr{K}$-indexed colimits)} is the colimit in $\rmm\catinfty^\mathscr{K}$:
    \[\mathrm{Stab}^\mathscr{K}_X(\ccal)^\otimes\coloneqq\colim(\ccal^\otimes \xrightarrow{X\otimes-}\ccal^\otimes\xrightarrow{X\otimes-}\ccal^\otimes\xrightarrow{X\otimes-}\cdots).\]
This construction can obviously be arranged into a functor $(\rmm\catinfty^\mathscr{K})_{\ccal^\otimes/}\to( \rmm\catinfty^\mathscr{K})_{\ccal^\otimes/}$. We let \[\Sigma^\infty_X\colon \ccal^\otimes \to \mathrm{Stab}^\mathscr{K}_X(\ccal)^\otimes\] we the canonical functor.
\end{definition}    

\begin{rem}
    If $\mathscr{K}\subseteq\mathscr{K}'$ are small collections of simplicial sets, then we have a natural equivalence 
    \[\mathrm{Stab}_X^\mathscr{K'}\circ \mc{P}_\mathscr{K}^\mathscr{K'} \xrightarrow{\sim} \mc{P}_\mathscr{K}^\mathscr{K'}\circ\mathrm{Stab}_X^\mathscr{K}\] of functors $\rmm\catinfty^\mathscr{K}\to \rmm\catinfty^\mathscr{K'}$.
\end{rem}
\begin{rem}When $\mathscr{K}=\Sigma$, and $X$ is the object $S^1_\ccal:=*\sqcup_* *\in \ccal$, we recover the Spanier-Whitehead category $\mathrm{SW}(\ccal)^\otimes$ associated to $\ccal^\otimes$. Furthermore, the colimit defining the Spanier-Witehead category can be computed in $\catinfty^\times$; if $\ccal^\otimes$ belongs to $(\catinfty^\mathrm{rex})^\otimes$, the Spanier-Witehead category also has finite colimits and can be seen as an object of  $(\catinfty^\mathrm{rex})^\otimes$ which coincides with the same colimit computed in $(\catinfty^\mathrm{rex})^\otimes$.
\end{rem}
\begin{definition}
    Let $\mathrm{A}^\otimes$ be in $\mathrm{MAdd}$. The category of \emph{bounded chain complexes up to chain homotopy equivalences} $\mathrm{K}^b(\mathrm{A})^\otimes$ is the category $\mathrm{SW}(\mathrm{K}^b_{\geqslant 0}(\mathrm{A}))^\otimes\in \rmm\catinfty^\mathrm{st}$.
\end{definition}
\begin{rem}
    Let $\mathrm{A}^\otimes$ be in $\mathrm{MAdd}$. The underlying $\infty$-category $\mathrm{K}^b(\mathrm{A})$ of $\mathrm{K}^b(\mathrm{A})^\otimes$ is the localization of $\ccal^b(\mathrm{A})$ with respect to chain homotopy equivalences. The monoidal structure is given by the dg-structure described above.
\end{rem}
The following proposition is essentially due to Robalo.
\begin{proposition}\label{universalStab}
    Let $\mathscr{K}$ be a small collection of simplicial sets. Let $\ccal^\otimes\in \rmm\catinfty^\mathscr{K}$ and let $X$ be an object of $\ccal$. Then, for any $\dcal^\otimes\in \rmm\catinfty^\mathscr{K}$ the functor 
    \[\Fun_{\rmm\catinfty^\mathscr{K}}(\mathrm{Stab}^\mathscr{K}_X(\ccal)^\otimes,\dcal^\otimes)\xrightarrow{-\circ \Sigma^\infty_X} \Fun_{\rmm\catinfty^\mathscr{K}}(\ccal^\otimes,\dcal^\otimes)\]
    is fully faithful. Its essential image is the full subcategory made of those maps $\ccal^\otimes \to \dcal^\otimes$ that send $X$ to a $\otimes$-invertible object.
\end{proposition}
\begin{proof}
We adapt several arguments from \cite[Section~4]{rob2}. The forgetful functor
\[\rmm\catinfty \to \catinfty\]
admits a left adjoint $F(-)^\otimes$ which assigns to an $\infty$-category $\dcal$, the free symmetric monoidal $\infty$-
category generated by $\dcal$. We let $\ccal_0^\otimes\coloneqq F(\Delta^0)^\otimes$ and we let $X_0\in \ccal_0^\otimes$ be the image of the canonical map $\Delta^0\to \ccal_0$.

Let $(\rmm\catinfty)^{X_0}_{\ccal_0^\otimes/}$
be the full subcategory of $(\rmm\catinfty)_{\ccal_0^\otimes/}$ spanned by the algebras $\ccal_0^\otimes\to \dcal^\otimes$ whose structure map sends $X_0$ to an invertible object. By \cite[Proposition~4.1]{rob2}, the inclusion functor \[(\rmm\catinfty)^{X_0}_{\ccal_0^\otimes/}\to (\rmm\catinfty)_{\ccal_0^\otimes/}\]
has a left adjoint $L_0$. By \cite[Proposition~4.2]{rob2}, the forgetful functor
\[(\rmm\catinfty)_{L_0(\ccal_0^\otimes)/}\to (\rmm\catinfty)_{\ccal^\otimes_0/}\]
is fully faithful and its essential image is $(\rmm\catinfty)^{X_0}_{\ccal^\otimes_0/}$. Through this equivalence, the functor $L_0$ corresponds to tensoring with $L_0(\ccal_0^\otimes)$.

By \cite[Proposition~4.21]{rob2}, the object $\Sigma^\infty_{X_0}(X_0)$ in $\mathrm{Stab}_{X_0}^\varnothing(\ccal_0)$ is $\otimes$-invertible. This means that $\Sigma^\infty_{X_0}$ factors through $L_0(\ccal_0^\otimes)$. We claim that the map $L_0(\ccal_0^\otimes)\to \mathrm{Stab}_{X_0}^\varnothing(\ccal_0)^\otimes$ is an equivalence. The proof goes as the proof of \cite[Corollary~4.24]{rob2}, but we include it for the reader's convenience. Our map factors as 
\[L_0(\ccal_0^\otimes)\xrightarrow{L_0\big(\Sigma^\infty_{X_0}\big)}L_0(\mathrm{Stab}_{X_0}^\varnothing(\ccal_0)^\otimes)\xrightarrow{\varepsilon}\mathrm{Stab}_{X_0}^\varnothing(\ccal_0)^\otimes\]
where $\varepsilon$ is the counit of the adjunction given by $L_0$ and the inclusion. As the image of $X_0$ in $\mathrm{Stab}_{X_0}^\varnothing(\ccal_0)$ is $\otimes$-invertible, the map $\varepsilon$ is an equivalence. Furthermore, as $L_0$ is a left adjoint functor, it commutes with colimits so that it commutes with $\mathrm{Stab}^\varnothing_{X_0}$ and through the equivalence $\mathrm{Stab}_{X_0}^\varnothing(L_0(\ccal_0))^\otimes\xrightarrow{\sim} L_0(\mathrm{Stab}_{X_0}^\varnothing(\ccal_0)^\otimes)$, the map $L_0(\Sigma^\infty_{X_0})$ becomes the canonical map $L_0(\ccal_0^\otimes)\to\mathrm{Stab}_{X_0}^\varnothing(L_0(\ccal_0))^\otimes$. This map is an equivalence because $X_0$ is $\otimes$-invertible in $L_0(\ccal_0^\otimes)$ so that all the maps involved in the colimit defining $\mathrm{Stab}_{X_0}^\varnothing(L_0(\ccal_0))^\otimes$ are equivalences proving our claim.

To prove the proposition, observe that the data of any object of $\ccal$ is equivalent to that of a map $\mc{P}_\varnothing^\mathscr{K}(\ccal_0^\otimes)\to \ccal^\otimes$ by the various universal properties of the involved functors. Hence, the data of $X$ yields a map $\mc{P}_\varnothing^\mathscr{K}(\ccal_0^\otimes)\to \ccal^\otimes$. Observe now that in $\rmm\catinfty^\mathscr{K}$, we have 
\[\mathrm{Stab}^\mathscr{K}_X(\ccal)^\otimes=\ccal^\otimes\sqcup_{\mc{P}_\varnothing^\mathscr{K}(\ccal_0^\otimes)}\mc{P}^\mathscr{K}_\varnothing(\mathrm{Stab}_{X_0}^\varnothing(\ccal_0)^\otimes).\]
Hence, $\Map_{\rmm\catinfty^\mathscr{K}}(\mathrm{Stab}_X^\mathscr{K}(\ccal)^\otimes,\dcal^\otimes)$ can be written as the pullback
\[\begin{tikzcd}
    &\{Y\in \dcal \mid Y \text{ is }\otimes-\text{invertible}\}\ar[d,hook] \\
\Map_{\rmm\catinfty^\mathscr{K}}(\ccal^\otimes,\dcal^\otimes) \ar[r,"F\mapsto F(X)"] & \dcal
\end{tikzcd}\]
which finishes the proof.
\end{proof}

\begin{corollary}
    The inclusion $\rmm\catinfty^\mathrm{st}\hookrightarrow\rmm\catinfty^\mathrm{rex}$ has a left adjoint given by the Spanier-Whitehead functor. 
\end{corollary}

\section{The universal property of Voevodsky motives over a perfect field}
If $X$ is a scheme, denote by $\mc{Z}(X)$ the abelian group of algebraic cycles on $X$.
\begin{definition}(Voevodsky)\label{defcorr} Let $k$ be a perfect field and let $\Lambda$ be a commutative ring.
    We define the monoidal additive category $\Sm^\mathrm{cor}(k,\Lambda)^\otimes$ whose objects are smooth $k$-schemes and whose morphisms are given by the abelian groups $c_\Lambda(-,-)$ of finite correspondences:
\[c_\Lambda(X,Y)\coloneqq c(X,Y)\otimes_\Z \Lambda\] \[c(X,Y)=\{\alpha\in \mc{Z}(X\times_k Y)\mid \alpha\to X\text{ is finite and dominant over an irreducible component}\}\]
where a cycle is said to have a certain property if it has nonzero coefficients only on subvarieties having this property; composition is given by the usual composition of correspondences and the monoidal structure is given by the cartesian product on objects and by the external product of cycles on morphisms. 
\end{definition}
\begin{constr}(Voevodsky)\label{constrDMgm} Let $k$ be a perfect field and let $\Lambda$ be a commutative ring. 
We have a monoidal functor $\gamma\colon \Sm_k^\times \to \Sm^\mathrm{cor}(k,\Lambda)^\otimes$ which maps a morphism to its graph. We now recall Voevodsky's category of geometric motives $\DM_\mathrm{gm}(k,\Lambda)$. Let $Z$ be the thick subcategory of $\mathrm{K}^b(\Sm^\mathrm{cor}(k,\Lambda))$ generated by the complexes of the form:
\[\gamma(\AAA_X)\xrightarrow{\gamma(p)} \gamma(X)\]
and 
\[\gamma(U\cap V)\xrightarrow{\gamma(j^U)\oplus \gamma(j^V)} \gamma(U)\oplus \gamma(V)\xrightarrow{\gamma(j_U)\oplus(-\gamma(j_V))} \gamma(X)\]
for $X$ smooth over $k$, $p\colon \AAA_X\to X$ the canonical projection, $(U,V)$ an open cover of $X$ and $j_U\colon U\to X$, $j_V\colon V\to X$, $j^U\colon U\cap V\to U$ and $j^V\colon U\cap V\to V$ the inclusions. 
After Voevodsky, we define the symmetric monoidal $\infty$-category $\DM^\eff_\mathrm{gm}(k,\Lambda)^\otimes$ of \emph{effective motives} as the idempotent-completion of the Verdier quotient of $\mathrm{K}^b(\Sm^\mathrm{cor}(k,\Lambda))^\otimes$ by $Z$, namely its localization at the class of morphism whose cofiber lies in $Z$. As $Z$ is a tensor-ideal, it underlies a monoidal $\infty$-category such that the localization functor from $\mathrm{K}^b(\Sm^\mathrm{cor}(k,\Lambda))^\otimes$ is symmetric monoidal. 

We let $M\colon \Sm_k^\times\to \DM^\eff_\mathrm{gm}(k,\Lambda)^\otimes$ be the composite monoidal functor.
The point at infinity of the projective line gives a decomposition 
\[M(\mb{P}^1_k)=\Lambda_k\oplus \Lambda_k(1)[2]\]
 where $\Lambda_k\coloneqq M(\Spec(k))$ and $\Lambda_k(1)$ exists by idempotent-completeness. We let 
 \[\DM_\mathrm{gm}(k,\Lambda)^\otimes=\mathrm{Stab}_{\Lambda_k(1)}^\mathrm{rex}(\DM^\eff_\mathrm{gm}(k,\Lambda))^\otimes.\]
\end{constr}

\begin{theorem}\label{universalprtyDM}
    Let $k$ be a perfect field, let $\Lambda$ be a commutative ring and let $\ccal^\otimes$ be a small symmetric monoidal idempotent-complete stable $\infty$-category. Denote by $M^\mathrm{cor}\colon \Sm^\mathrm{cor}(k,\Lambda)^\otimes\to \DM_\mathrm{gm}(k,\Lambda)^\otimes$ the canonical functor. Then, the functor
    \[\Fun_{\rmm\catinfty^\mathrm{st}}(\DM_\mathrm{gm}(k,\Lambda)^\otimes,\ccal^\otimes)\xrightarrow{-\circ M^\mathrm{cor}} \Fun_{\rmm\catinfty^\Sigma}(\Sm^\mathrm{cor}(k,\Lambda)^\otimes,\ccal^\otimes)\] 
    is fully faitful and its essential image is the full subcategory of those symmetric monoidal $\infty$-functors $F\colon \Sm^\mathrm{cor}(k,\Lambda)^\otimes\to \ccal^\otimes$ such that
    \begin{enumerate}
        \item for any smooth $k$-scheme $X$ over $k$, letting $p\colon \AAA_X\to X$ be the canonical projection, the map $F(p)\colon F(\AAA_X)\to F(X)$ is an equivalence.
        \item If $(U,V)$ is an open cover of a smooth $k$-scheme $X$ and $j_U\colon U\to X$, $j_V\colon V\to X$, $j^U\colon U\cap V\to U$ and $j^V\colon U\cap V\to V$ are the inclusions, the triangle \[F(U\cap V)\xrightarrow{F(j^U)\oplus F(j^V)} F(U)\oplus F(V)\xrightarrow{F(j_U)\oplus(-F(j_V))} F(X)\]
        is exact.
        \item The object $\mathrm{Fib}(F(\mb{P}^1_k)\to F(\Spec(k)))$ is $\otimes$-invertible.
    \end{enumerate}
\end{theorem}
\begin{proof}
    Combine \Cref{universalKb,universalStab} with the universal property of symmetric monoidal localizations \cite[Theorem~1.2.1]{hinich}.
\end{proof}


\section{The real realization}
The singular chain complex with $\mb{F}_2$-coefficients may be seen as a functor:
\[C_*(-,\F_2)\colon \Sm_{\R}^\times\to\D^b(\F_2)^\otimes.\]
To show that $C_*(-,\F_2)$ factor through Voevodsky motives, we must first show that it factors through $\Sm^\mathrm{cor}(\R,\F_2)^\otimes$. 
At the level of homology which is in fact the level of the homotopy category because $\F_2$ is a field, there is a somehow natural way of defining maps \[ c_{\F_2}(X,Y)\to \Hom_{\F_2}(\HH_n(X(\R),\F_2),\HH_n(Y(\R),\F_2))\]
(see \Cref{concretedescription} below), but it is then not obvious at all (at least to the author of this note) how to show that these maps are compatible with composition even without thinking about higher coherence. Circumventing this issue will require a detour through the theory of quadratic forms and the use of Jacobson's theorem (\Cref{jacobson} below).

\begin{definition}(Knebusch) Let $X$ be a scheme. 
\begin{enumerate}
    \item Let $\mathrm{GW}'(X)$ be the Grothendieck ring associated to the semi-ring of \emph{bilinear spaces} i.e. equivalence classes of locally free sheaves $\ecal$ on $X$ equipped with a non-degenerate symmetric bilinear form \[\beta\colon \ecal\times \ecal \to \ocal_X,\] 
where the semi-ring structure is given by orthogonal sum and tensor product. A bilinear space $(\ecal,\beta)$ is said to be \emph{metabolic} if it admits a Lagrangian, i.e. a sub-bundle $\mc{V}$ which coincides with its own orthogonal bundle\footnote{Metabolic spaces are the analog of hyperbolic spaces over a base.}. The \emph{Witt ring} $W(X)$ is the quotient of $\mathrm{GW}'(X)$ by the ideal generated by metabolic bilinear spaces. 
\item There is a natural map $W(X)\to \F_2^{\pi_0(X)}$ given by the rank modulo $2$. The kernel of this map is called the \emph{fundamental ideal} of $W(X)$ and is denoted by $I(X)$. For $n\geqslant 0$ an integer, we denote by $I^n(X)$ the $n$-th power of the fundamental ideal.
\end{enumerate}
\end{definition}
\begin{constr}(Jacobson) Let $X$ be a finite-type $\R$-scheme. The \emph{global signature map} is the map 
    \[\sigma\colon W(X)\to \mathrm{H}^0(X(\R),\Z)\]
    which sends a bilinear space $(\ecal,\beta)$ to the locally constant function on $X(\R)$ given by mapping a real point $x$ to the signature of the restriction $\beta_x$ of $\beta$ to $x$. 

    In \cite[Definition~8.3]{jacobson}, it is explained why this map lands in $\mathrm{H}^0(X(\R),2^n\Z)$ when restricted to $I^n(X)$ for any integer $n\geqslant0$, whence a family of maps \[\left(\sigma_n:=\frac{1}{2^n}\sigma \colon I^n(X) \to \mathrm{H}^0(X(\R),\Z)\right)_{n\geqslant 0}.\] 
    Multiplication by $2=\langle 1,1\rangle \in I(X)$ induces for all $n\geqslant 0$ a map $I^n(X)\xrightarrow{\times 2} I^{n+1}(X)$ and the $\sigma_n$ form a compatible family of maps $\sigma_{n+1}\circ (\times 2)=\sigma_n$. 
    

Let $\overline{I}^n(X)=I^n(X)/I^{n+1}(X)$ for $n\geqslant 0$; the above family of map yields a family of maps \[\left(\overline{\sigma}_n\colon \overline{I}^n(X) \to \mathrm{H}^0(X(\R),\F_2)\right)_{n\geqslant 0}\]
and these maps are compatible with one another: multiplication by $2$ still gives a map $\overline{I}^n(X)\to \overline{I}^{n+1}(X)$ and $\overline{\sigma}_{n+1}\circ (\times 2)=\overline{\sigma}_n$.    
\end{constr}
\begin{constr}\label{constrjacmap}(Jacobson) Let $X$ be a finite-type $\R$-scheme. We let 
    \[\iota\colon X(\R)\hookrightarrow X\]
    be the inclusion which is a continuous map when the right hand side is endowed with the Zariski topology. For $n\geqslant 0$, let $\overline{\mc{I}}^n$ be the Zariski sheaf associated to the presheaf $U\mapsto \overline{I}^n(U)$; the map $\overline{\sigma}_n$ induces a map of sheaves
    \[s_n\colon \overline{\mc{I}}^n\to \iota_*\F_2.\]
    Letting $\overline{\mc{I}}^\infty\coloneqq \colim_n \ \overline{\mc{I}}^n$, as the $s_n$ are compatible with one another, we get a map \[s_\infty\colon \overline{\mc{I}}^\infty\to \iota_*\F_2.\]
\end{constr}

Denote by $C^*(-,\F_2)$ the complex of singular cochains. The above construction will be relevant for us by using the following result:
\begin{theorem}\label{jacobson} \cite[Theorem~8.5]{jacobson} Let $X$ be a finite-type $\R$-scheme. Then the map $s_\infty$ is an isomorphism of Zariski sheaves. In particular, we have a natural isomorphism \[\mathbf{s_\infty}\colon R\Gamma(-, \overline{\mc{I}}^\infty)\to C^*(-(\R),\F_2).\] 
of functors from $(\Sm_{\R}^\op)^\times$ to $\D^b(\F_2)^\otimes$. 
\end{theorem}
Using the above theorem, our strategy is now to show that the functors \[R\Gamma(-,\overline{\mc{I}}^n)\colon (\Sm_\R^\op )^\times\to \D^b(\F_2)^\otimes\] for $n\geqslant 0$ factor through $(\Sm^\mathrm{cor}(\R,\F_2)^\op)^\otimes$. 
This is a fairly classical statement in similar contexts (see \cite[Chapter~1, Lemma~2.41]{MWmotives}). To do this, we will first prove that the sheaves $\overline{\mc{I}}^n$ admits an additional structure of sheaves with transfers. Endowed with this further structure, we will show that they are effective motives over $\R$.

\begin{constr}(Voevodsky) 
    Let $k$ be a perfect field and let $\Lambda$ be a commutative ring. 
    \begin{enumerate}
        \item  The symmetric monoidal abelian category of \emph{spectral presheaves with transfers} is the symmetric presentably monoidal stable $\infty$-category \[\mathrm{PSh}^\mathrm{tr}(k,\Lambda)^\otimes\coloneqq\Fun_\Sigma(\Sm^\mathrm{cor}(k,\Lambda)^\op,\mathrm{Sp})^\otimes=\Ind(\mathrm{K}^b(\Sm^\mathrm{cor}(k,\Lambda)^\otimes))\] with $\mathrm{Sp}$ the category of spectra (tensor product is given by Day convolution or by the tensor product on the Ind-category which coincide by \cite[Remark~4.8.1.13]{ha}).
        \item A presheaf with transfers $F$ is a Nisnevich sheaf if $F\circ \gamma$ is a Nisnevich sheaf. Denote by $\Sh_\mathrm{Nis}^\mathrm{tr}(k,\Lambda)$ the full subcategory of \emph{Nisnevich sheaves with transfers}. It is a Bousfield localization of $\mathrm{PSh}^\mathrm{tr}(k,\Lambda)$ compatible with the tensor structure in the sense of \cite[Example~2.2.1.7]{ha} and therefore by \cite[Proposition~2.2.1.9]{ha}, it underlies a symmetric presentably monoidal stable $\infty$-category $\Sh_\mathrm{Nis}^\mathrm{tr}(k,\Lambda)^\otimes$ and there is a symmetric monoidal enhancement 
    \[L_\mathrm{Nis}\colon \PSh^\mathrm{tr}(k,\Lambda)^\otimes\to \Sh_\mathrm{Nis}^\mathrm{tr}(k,\Lambda)^\otimes\]
     of the localization functor.
     \item Let $\Sch_k$ (resp. $\Sch_k^\mathrm{prop}$) be the category of finite type $k$-schemes (resp. finite type $k$-schemes with proper maps as morphisms). We can define two symmetric monoidal functors 
     \[\Lambda^\mathrm{tr}\colon \Sch_k^\times\to \Fun_\Sigma(\Sm^\mathrm{cor}(k,\Lambda),\mathrm{Ab})^\otimes\]
     \[\Lambda^\mathrm{tr}_c\colon (\Sch_k^\mathrm{prop})^\times\to \Fun_\Sigma(\Sm^\mathrm{cor}(k,\Lambda),\mathrm{Ab})^\otimes\]
     which in the case $\Lambda=\Z$ are denoted by $L$ and $L^c$ in \cite[Chapter~5, Section~4.1]{orange}. For any smooth $k$-scheme $U$ and any $k$-scheme of finite type $X$,
\begin{align*}
    \Lambda^\mathrm{tr}(X)(U)=&\{\alpha \in \mc{Z}(U\times_k X)\otimes_\Z\Lambda\mid
    \\& \alpha\to U\text{ is finite and dominant over a connected component}\}
    \\
    \Lambda^\mathrm{tr}_c(X)(U)=&\{\alpha \in \mc{Z}(U\times_k X)\otimes_\Z\Lambda\mid \\
    &\alpha\to U\text{ is quasi-finite and dominant over a connected component}\}
\end{align*}
Composing those with the Eilenberg-Maclane functor $H\colon \mathrm{Ab}^\otimes\to\mathrm{Sp}^\otimes$ we get two functors $H\Lambda^\mathrm{tr}$ and $H\Lambda^\mathrm{tr}_c$.
The full subcategory of objects which are local with respect to the maps $H\Lambda^\mathrm{tr}(\AAA_X)\to H\Lambda^\mathrm{tr}(X)$ for $X$ smooth over $k$ underlies a symmetric monoidal $\infty$-category $\DM^\eff(k,\Lambda)^\otimes$ (the category of \emph{effective motives}) and there is a symmetric monoidal $\infty$-functor
\[L_{\AAA}\colon \Sh^\mathrm{tr}_\mathrm{Nis}(k,\Lambda)^\otimes\to \DM^\eff(k,\Lambda)^\otimes\]
which enhances the localization functor. The name of this category is justified by Voevodsky's embedding theorem (see \Cref{embeddingthm} below).
\end{enumerate}
\end{constr}

 The critical point is now that if $n\geqslant 0$, the sheaf $\overline{\mc{I}}^n$ on $\Sm_\R$ admits a flabby resolution called the \emph{Gersten resolution}:
\[0\to \overline{\mc{I}}^n\to \bigoplus_{x\in (-)^{(0)}}\overline{I}^n(k(x))\to \bigoplus_{x\in (-)^{(1)}}\overline{I}^{n-1}(k(x))\to \cdots \to \bigoplus_{x\in (-)^{(p)}}\overline{I}^{n-p}(k(x))\to \cdots\]
where $Y^{(m)}$ is the set of codimension $m$ points of a scheme $Y$ and if $q<0$ and $F$ is a field, $\overline{I}^{q}(F)=W(F)/I(F)$; this is actually a very subtle result as it follows from \cite[Corollary~0.5]{milnorchowrevisited} which is a consequence of the Milnor conjecture for K-theory. This resolution allows us to endow the $\overline{\mc{I}}^n$ with a structure of a sheaf with transfers using the following proposition. 

\begin{proposition}\label{deglisefactor}\cite[Corollaire~4.3.6, Remarque~4.3.10]{thesedeglise}
    Let $k$ be a perfect field. Let $\mathrm{F}^*$ be a cycle module over $k$ in the sense of \cite{rost}. For any integer $n$, let \[\mc{F}^n=\ker\left(\bigoplus_{x\in (-)^{(0)}}\mathrm{F}^n(k(x))\to \bigoplus_{x\in (-)^{(1)}}\mathrm{F}^{n-1}(k(x))\right).\] Then $\mc{F}^n$ naturally extends to a presheaf with transfers. Furthermore, this presheaf belongs to $\DM^\eff(k,\Lambda)$.
\end{proposition}\begin{constr}\label{constrtransfert}
 Using \cite{arason}, we have a cycle module $\overline{I}^*$ over any perfect field $k$ whose value on an extension $F/k$ of finite transcendence degree is $\overline{I}^*(F)$. This yields by \Cref{deglisefactor} an object $\overline{\mc{I}}^n_\mathrm{tr}$ in $\DM^\eff(\R,\F_2)$ such that \[\overline{\mc{I}}^n=\overline{\mc{I}}^n_\mathrm{tr}\circ \gamma.\] 
Hence by adjunction, we have an isomorphism 
\[\map_{\DM^\eff(\R,\F_2)}(H\F_2^\mathrm{tr}(-),\overline{\mc{I}}^n_\mathrm{tr})\to R\Gamma(-,\overline{\mc{I}}^n)\]
of functors from $(\Sm_{\R}^\op)^\times$ to $\D^b(\F_2)^\otimes$. The left hand side factors through $(\Sm^\mathrm{cor}(\R,\F_2)^\otimes)^\op$ as the functor $H\F_2^\mathrm{tr}(-)$ factors through $\Sm^\mathrm{cor}(\R,\F_2)^\otimes$. Thus by Jacobson's Theorem (\Cref{jacobson}), we get a factorization $C_\mathrm{{cor}}^*$ of $C^*(-,\F_2)$ through $\Sm^\mathrm{cor}(\R,\F_2)^\otimes$. Applying $\map(-,\F_2)$, we get a a factorization $C^\mathrm{{cor}}_*$ of $C_*(-,\F_2)$ through $\Sm^\mathrm{cor}(\R,\F_2)^\otimes$. 
\end{constr}

\begin{rem}\label{concretedescription} We can give a more concrete description of the factorization on homology groups. 
Let $f\colon X'\to X$ be a proper map of real topological manifolds of pure dimensions. Then, the \emph{transfer map} $f^!$ is defined as the composition
    \[\HH_*(X(\R),\F_2)\xrightarrow{\mathrm{PD}^{-1}} \HH^{\dim(X)-*}_c(X(\R),\F_2)\xrightarrow{f^*}\HH^{\dim(X)-*}_c(X'(\R),\F_2)\xrightarrow{\mathrm{PD}}\HH_{*+\dim(X')-\dim(X)}(X'(\R),\F_2)\]
    where $\mathrm{PD}$ is the Poincaré duality map from cohomology with compact support to homology.Let $X,Y$ be smooth real algebraic varieties. If $Z\hookrightarrow X\times_k Y$ is a closed subscheme such that the map $p \colon Z\to X$ is finite and dominant, letting $q\colon Z\to Y$ be the canonical map, we have a map \[\varphi_Z\colon \HH_*(X(\R),\F_2)\xrightarrow{p^!}\HH_*(Z(\R),\F_2)\xrightarrow{q_*}\HH_*(Y(\R),\F_2).\]
These assemble into a map 
\[ c(X,Y)\to \Hom_{\F_2}(\HH_*(X(\R),\F_2),\HH_*(Y(\R),\F_2))\]
which we invite the reader to check that it is the one given by our functor. With this description it is however completely unclear how to show compatibility with composition. 
\end{rem}

\begin{rem}
    Another way of constructing our functor is to use the existence of a Gersten complex for real cohomology of real schemes from \cite{jinxie}. Their approach also relies on the existence of the Gersten complex for $\overline{\mc{I}}$ and therefore on the Milnor conjecture.
\end{rem}
  

\begin{proposition}
    The functor $C^\mathrm{{cor}}_*\colon \Sm^\mathrm{cor}(\R,\F_2)^\otimes\to \D^b(\F_2)^\otimes$ factors through $\DM_\mathrm{gm}(\R,\F_2)^\otimes$. Hence, we have a \emph{real realization functor} 
    \[\rho_{\R}\colon\DM_\mathrm{gm}(\R,\F_2)^\otimes\to \D^b(\F_2)^\otimes.\]
\end{proposition}
\begin{proof}
    This is a direct consequence of \Cref{universalprtyDM}.
\end{proof}

\begin{rem}\label{furealismyreal}
    On Chow motives, the functor $\rho_\R$ induces the functor $H\R^*$ constructed in \cite[Section~11.2(iv)]{furealreal}. This can be proved using \Cref{concretedescription}. 
\end{rem}

Let us now recall Voevodsky's constructions of the motivic analogs of homology, Borel Moore homology, cohomology and cohomology with compact support. We will then show that their respective images through the real realization are the usual version of those (co)homology theories.

\begin{constr}\label{embeddingthm}
    (Voevodsky's Embedding Theorem) The objects defining the localization $\DM^\eff_\mathrm{gm}(k,\Lambda)$ of $\mathrm{K}^b(\Sm^\mathrm{cor}(k,\Lambda))$ are sent to zero through the canonical functor $\mathrm{K}^b(\Sm^\mathrm{cor}(k,\Lambda))\to \PSh^\mathrm{tr}(k,\Lambda)$ so that we get a symmetric monoidal exact $\infty$-functor 
\[i\colon \DM^\eff_\mathrm{gm}(k,\Lambda)^\otimes\to \DM^\eff(k,\Lambda)^\otimes.\]
By \cite[Lemma~2.61]{Pstragowski}, we have a canonical equivalence \[\mathrm{PSh}^\mathrm{tr}(k,\Lambda)\xrightarrow{\sim} \D(\Fun_\Sigma(\Sm^\mathrm{cor}(k,\Lambda)^\op,\mathrm{Ab})))\]
we can then check that our construction is an $\infty$-categorical enhancement of that of \cite[Chapter~5, Section~3.1]{orange}. Hence, the functor $i$ is fully faithful by \cite[Chapter~5, Theorem~3.2.6]{orange}.  Note that it sends $M(X)$ to $L_{\AAA}(\Lambda^\mathrm{tr}(X))$. 
\end{constr}

\begin{definition}Let $k$ be a perfect field and let $\Lambda$ be a commutative ring.
    Assume that $k$ allows resolutions of singularities. 
    
    \begin{enumerate}
        \item By \cite[Chapter~5, Corollaries~4.1.4 \&~4.1.6]{orange}, the functors $L_{\AAA} \circ H\Lambda^\mathrm{tr}$ and $L_{\AAA} \circ H\Lambda^\mathrm{tr}_c$ land in $\DM^\eff_\mathrm{gm}(k,\Lambda)$. We denote by $M$ and $M^\mathrm{BM}$ their respective compositions with $\Sigma^\infty_{\Lambda(1)}$\footnote{The functor $M$ restricts to the one defined before on smooth $k$-schemes.}
        \item By \cite[Chapter~5, Corollary~4.3.4]{orange}, the internal Hom of $\DM^\eff(k,\Lambda)$ induces an internal Hom functor on $\DM_\mathrm{gm}(k,\Lambda)$. We let \[\mathbb{D}\coloneqq\underline{\Hom}_{\DM_\mathrm{gm}(k,\Lambda)}(-,\Lambda)\] be the \emph{Poincaré duality functor}. We set $h\coloneqq\mathbb{D}\circ M$ and $h_c\coloneqq\mathbb{D}\circ M^\mathrm{BM}$.
    \end{enumerate}

\end{definition}

Denote by $C_*^\mathrm{BM}(-,\F_2)$ the complex of singular chains with closed support and by $C^*_c(-,\F_2)$ the complex of singular cochains with closed support. 
\begin{proposition}
    The functor $\rho_\R$ sends $M$ (resp. $M^\mathrm{BM}$, resp. $h$, resp. $h_c$) to $C_*(-,\F_2)$ (resp. $C_*^\mathrm{BM}(-,\F_2)$, resp. $C^*(-,\F_2)$, resp. $C^*_c(-,\F_2)$).
\end{proposition}
\begin{proof}
    By definition, $C_*(-,\F_2)$ and $\rho_\R\circ M$ coincide over $\Sm_\R$. Observe that if $p\colon X'\to X$ is proper birational and $Z\to X$ is a closed embedding such that $p^{-1}(X\setminus Z)\to X\setminus Z$ is an isomorphism, the canonical triangle
    \[M(p^{-1}(Z))\to M(X')\oplus M(Z)\to M(X)\]
    is exact by \cite[Chapter~5, Proposition~4.1.3]{orange}. As the same is true for the analogous triangle with $C_*(-,\F_2)$, we can show by induction on $d$ that the functors $C_*(-,\F_2)$ and $\rho_\R\circ M$ coincide over the category $\Sch^{\leqslant d}_\R$ of finite-type $\R$-schemes of dimension $d$ or less which yields the result over $\Sch_\R$.

    Since $C_*(-,\F_2)$ and $C_*^\mathrm{BM}(-,\F_2)$ (resp. $M$ and $M^\mathrm{BM}$) coincide over proper schemes, we get that $C_*^\mathrm{BM}(-,\F_2)$ and $\rho_\R\circ M^\mathrm{BM}$ coincide over proper schemes. The localization triangle \cite[Chapter~5, Proposition~4.1.5]{orange} for the Borel-Moore motive and the localization triangle for cohomology with compact support yield that they coincide over $\Sch_\R^\mathrm{prop}$.

    Since the functor $\rho_\R$ is symmetric monoidal, we have a natural transformation
    \[\rho_\R\circ \mathbb{D}\to \mathbb{D}\circ \rho_\R\]
    The result will follow if we can show that this map is an equivalence. As both maps commute with finite limits, finite colimits and twists and as $\DM_\mathrm{gm}(\R,\F_2)$ is generated, under those operations, by the $M(X)$ for $X$ proper smooth over $\R$ by \cite[Chapter~5, Corollary~3.5.5]{orange}, it suffices to show that both maps coincide over those objects. It is then a consequence of Poincaré duality for motives \cite[Chapter~5, Theorem 4.3.2]{orange} and Poincaré duality for singular homology. 
\end{proof}
\section{The weight filtration on real singular homology}
We will now reconstruct the weight filtration on singular homology. To that end we recall the theory of weights and the weight structure on motives. 
\begin{definition}(Bondarko) 
    Let $\ccal$ be a stable $\infty$-category. A full subcategory $\hcal\subseteq \ccal$ is the \emph{heart of a bounded weight structure} if
    \begin{enumerate}
        \item it is closed under finite coproducts and retracts,
        \item it generates $\ccal$ under finite limits and colimits,
        \item the mapping spectrum $\map_{\ccal}(X, Y )$ is connective for all $X, Y \in\hcal$.
    \end{enumerate}
If $\ccal^\otimes\in \rmm\catinfty^\mathrm{st}$ is a monoidal enhancement of $\ccal$ we say that $\hcal$ is the \emph{heart of a bounded compatible
weight structure} if it furthermore contains the tensor unit and is closed under tensor products.
\end{definition}
\begin{ex}(Bondarko) Let $k$ be a perfect field of characteristic $0$ and let $\Lambda$ be a commutative ring. The heart of the \emph{Chow weight structure} is the subcategory $\mathrm{Chow}_\infty(k,\Lambda)$ of $\DM_\mathrm{gm}(k,\Lambda)$ generated under finite coproducts and retract by the $M(X)(n)[2n]$ for $X$ projective smooth over $k$ and $n\in \Z$. It indeed generates $\DM_\mathrm{gm}(k,\Lambda)$ because of resolutions of singularities and the relevant mapping spaces are connective by \cite[Proposition~3.1.1]{bondarkoZ[1/p]-motivic}. It is compatible with the monoidal structure on $\DM_\mathrm{gm}(k,\Lambda)$ because \[M(X)(n)[2n]\otimes M(Y)(m)[2m]=M(X\times_k Y)(n+m)[2(n+m)].\]
The homotopy category $\mathrm{Chow}_\infty(k,\Lambda)$ is Grothendieck's category $\mathrm{Chow}(k,\Lambda)$ of \emph{Chow motives}. 
\end{ex}



We now recall a lemma of used by Sosnilo, and Aoki in the monoidal context, to produce an $\infty$-categorical enhancement of Bondarko's weight complex functor. 

\begin{lemma}\label{aokiji}(Sosnilo-Aoki, compare with \cite[Proposition~3.27]{bgv} and \cite[Proposition~3.3]{sosnilonegativeKtheory})
    Let $\ccal^\otimes$ be a small stable symmetric monoidal $\infty$-category and let $\hcal\subseteq \ccal$ be the heart of a bounded compatible weight structure. There is an equivalence
    \[\ccal^\otimes \xrightarrow{\sim} \mathrm{SW}\circ \mc{P}^\mathrm{rex}_\Sigma(\hcal)^\otimes\]
    in $\rmm\catinfty^\mathrm{st}.$
In particular, for any stable monoidal $\infty$-category $\dcal^\otimes$, the natural functor
    \[\Fun_{\mathrm{MCat}^\mathrm{st}_\infty}(\ccal^\otimes,\dcal^\otimes)\to \Fun_{\mathrm{MCat}^\Sigma_\infty}(\hcal^\otimes,\dcal^\otimes)\]
    is an equivalence.
\end{lemma}
\begin{proof}
    By \cite[Lemma~4.2]{aoki}, there is an equivalence in $\rmm\catinfty^\mathrm{st}$ between $\ccal^\otimes$ and the full subcategory of $\mathrm{Sp}(\Fun_\Sigma(\hcal^\op,\scal_*))^\otimes$ generated by $\hcal$ under finite limits and colimits. This subcategory is equivalent to $\mathrm{SW}(\mc{P}^\mathrm{rex}_\Sigma(\hcal))^\otimes$.
\end{proof}
\begin{constr}\label{filrealreal} Let $\F$ be a field. Consider the symmetric monoidal (ordinary) category $\mathrm{Fil}(\mathrm{C}^b(\F))^\otimes$ of filtered bounded chain complexes of vector spaces over $F$ with a bounded filtration; its the monoidal product is given by 
\[F_n(C\otimes D):=\sum\limits_{i+j=n}F_i(C)\otimes F_j(D)\subseteq C\otimes D.\]
The \emph{filtered derived category} $\D^b\mathrm{F}(\F)^\otimes$ defined as the symmetric monoidal $\infty$-categorical localization (which exists thanks to \cite[Theorem~1.2.1]{hinich}) of $\mathrm{Fil}(\mathrm{C}^b(F))^\otimes$ at quasi-isomorphisms i.e. morphism which induce an isomorphism on the $\mathrm{E}_1$-page of the spectral sequence associated to the filtered complex (here we use the conventions of \cite[Section~1.1]{mccrory1} for filtered chain complexes). 

Any bounded chain complex can be endowed with the canonical filtration \cite[§~1.4.6]{hodgeII}. We claim that this can be promoted into a map in $\mathrm{M}\catinfty^\Sigma$ 
\[F^\mathrm{can}\colon \D^b(\F)^\otimes\to \D^b\mathrm{F}(\F)^\otimes.\]
Indeed, the inclusions 
\[\theta_n\colon \sum\limits_{i+j=n}F^\mathrm{can}_i(C)\otimes F_j^\mathrm{can}(D)\subseteq F^\mathrm{can}_n(C\otimes D)\]
for $n\in \Z$ yield a lax-monoidal functor $\mathrm{C}^b(\F)^\otimes\to \mathrm{Fil}(\mathrm{C}^b(\F))^\otimes$ which sends qusi-isomorphisms to quasi-isomorphisms and thus yield a lax monoidal functor $F^\mathrm{can}$ on the localizations by \cite[Theorem~1.2.1]{hinich}; because of the Künneth formula, the $\theta_n$s induce quasi-isomorphisms on the associated graded complexes (ans therefore on the $E_1$-page of the associated spectral sequences) and hence $F^\mathrm{can}$ is indeed symmetric monoidal.

On the other hand, using \cite[Theorem~1.2.1]{hinich} again, the forgetful functor can be promoted into a map in $\mathrm{M}\catinfty^\mathrm{rex}$ 
\[\varphi\colon \D^b\mathrm{F}(\F_2)^\otimes\to \D^b(\F)^\otimes\]
and the functor $\varphi\circ i$ is the identity of $\D^b(\F)^\otimes$ (this can be checked at the level of chain complexes).

We will use this for $\F=\F_2$. Restricting $F^\mathrm{can}\circ \rho_{\R}$ to $\mathrm{Chow}_\infty(\R,\F_2)^\otimes$ yields a map in $\mathrm{MCat}_\infty^\Sigma$, which can be extended by \Cref{aokiji} to a map 
\[\rho_\R^\mathcal{W}\colon\mathrm{DM}_\mathrm{gm}(\R,\F_2)^\otimes\to \D^b\mathrm{F}(\F_2)^\otimes\]
in $\mathrm{M}\catinfty^\mathrm{rex}$ such that $\varphi\circ \rho_\R^\mathcal{W}=\rho_\R.$ We call $\rho_\R^\mathcal{W}$ the \emph{filtered real realization}.
\end{constr}
We can now finally define the weight filtrations on real homology, cohomology, Borel-Moore homology and cohomology with compact support.

\begin{corollary}\label{mainthm}
Set $WC_*\coloneqq\rho_\R^\mc{W}\circ M$, $WC_*^\mathrm{BM}\coloneqq\rho_\R^\mc{W}\circ M^\mathrm{BM}$, $WC^*\coloneqq\rho_\R^\mc{W}\circ h$ and $WC^*_c\coloneqq\rho_\R^\mc{W}\circ h_c$. These functors refine homology, Borel-Moore homology, cohomology and cohomology with compact support of the real locus and yield weight spectral sequences that converge to those cohomology theories. Furthermore, \begin{enumerate}
        \item if $p\colon X'\to X$ is proper birational and $Z\to X$ is a closed embedding such that $p^{-1}(X\setminus Z)\to X\setminus Z$ is an isomorphism, we have canonical exact triangles:
    \[WC_*(p^{-1}(Z))\to WC_*(X')\oplus WC_*(Z)\to WC_*(X).\]
    \[WC_*^\mathrm{BM}(p^{-1}(Z))\to WC_*^\mathrm{BM}(X')\oplus WC_*^\mathrm{BM}(Z)\to WC_*^\mathrm{BM}(X)\]
    in $\D^b\mathrm{F}(\F_2)$.
    \item if $Z\to X$ is a closed immersion, we have a canonical exact triangle 
    \[WC_*^\mathrm{BM}(Z)\to WC_*^\mathrm{BM}(X)\to WC_*^\mathrm{BM}(U)\]
    in $\D^b\mathrm{F}(\F_2)$.
    \end{enumerate}
Finally, our functors $WC_*$ and $WC_*^\mathrm{BM}$are $\infty$-categorical enhancements of those of \cite{mccrory1,mccrory2}.
\end{corollary}
\begin{proof}
Both properties follow from the corresponding properties of $M$ and $M^\mathrm{BM}$ from \cite[Chapter~5, Propositions~4.1.3 \&~4.1.5]{orange}. The fact that our construction enhances that of \cite{mccrory1,mccrory2} comes from the uniqueness result in \cite[Theorem 1.1]{mccrory1} for Borel-Moore homology. In the case of a smooth scheme, the uniqueness for ordinary homology is obtained by Poincaré duality at the level of filtered complexes which is \cite[Theorem 2.4]{mccrory2}. In the general case, it follows from the uniqueness result of \cite[Theorem 7.1]{mccrory2}.
\end{proof}

	\bibliographystyle{alpha}
	\bibliography{biblio.bib}

\newcommand{\etalchar}[1]{$^{#1}$}
\begin{thebibliography}{BCKW19}

\bibitem[Aok20]{aoki}
Ko~Aoki.
\newblock The weight complex functor is symmetric monoidal.
\newblock {\em Adv. Math.}, 368:9, 2020.
\newblock Id/No 107145.

\bibitem[Ara75]{arason}
J{\'o}n~Kr. Arason.
\newblock Cohomologische {Invarianten} quadratischer {Formen}.
\newblock {\em J. Algebra}, 36:448--491, 1975.

\bibitem[Ayo10]{ayobetti}
Joseph Ayoub.
\newblock Notes on {Grothendieck}'s operations and {Betti} realizations.
\newblock {\em J. Inst. Math. Jussieu}, 9(2):225--263, 2010.

\bibitem[Ayo14]{ayo14}
Joseph Ayoub.
\newblock La r{\'e}alisation {\'e}tale et les op{\'e}rations de grothendieck.
\newblock {\em Annales {S}cientifiques de l'{\'E}cole {N}ormale
  {S}up{\'e}rieure}, 47:1--145, 2014.

\bibitem[BC20]{cisinski-bunke2}
Ulrich Bunke and Denis-Charles Cisinski.
\newblock A universal coarse {{\(K\)}}-theory.
\newblock {\em New York J. Math.}, 26:1--27, 2020.

\bibitem[BCD{\etalchar{+}}25]{MWmotives}
Tom Bachmann, Baptiste Calm{\`e}s, Fr{\'e}d{\'e}ric D{\'e}glise, Jean Fasel,
  and Paul~Arne {\O}stv{\ae}r.
\newblock {\em Milnor-Witt motives}, volume 1572 of {\em Mem. Am. Math. Soc.}
\newblock Providence, RI: American Mathematical Society (AMS), 2025.

\bibitem[BCKW19]{cisinski-bunke}
Ulrich Bunke, Denis-Charles Cisinski, Daniel Kasprowski, and Christoph Winges.
\newblock Controlled objects in left-exact {$\infty$}-categories and the
  novikov conjecture.
\newblock Available at \url{https://arxiv.org/abs/1911.02338}, 2019.

\bibitem[BGV24]{bgv}
Federico Binda, Martin Gallauer, and Alberto Vezzani.
\newblock Motivic monodromy and p-adic cohomology theories.
\newblock Preprint, available at \url{https://arxiv.org/abs/2306.05099}, 2024.

\bibitem[Bon11]{bondarkoZ[1/p]-motivic}
Mikhail~V. Bondarko.
\newblock {$\mathbb{Z}[\frac{1}{p}]$}-motivic resolution of singularities.
\newblock {\em Compos. Math.}, 147(5):1434--1446, 2011.

\bibitem[CD16]{em}
Denis-Charles Cisinski and Fr{\'e}d{\'e}ric D{\'e}glise.
\newblock {\'E}tale motives.
\newblock {\em Compositio Mathematica}, 152(3):277--427, 2016.

\bibitem[D{\'e}g02]{thesedeglise}
Fr{\'e}d{\'e}ric D{\'e}glise.
\newblock {\em Modules homotopiques avec transferts et motifs
  g{\'e}n{\'e}riques}.
\newblock PhD thesis, Universit{\'e} Paris 7, 2002.

\bibitem[Del71]{hodgeII}
Pierre Deligne.
\newblock Th{\'e}orie de {Hodge}. {II}. ({Hodge} theory. {II}).
\newblock {\em Publ. Math., Inst. Hautes {\'E}tud. Sci.}, 40:5--57, 1971.

\bibitem[Del74]{hodgeIII}
Pierre Deligne.
\newblock Th{\'e}orie de {Hodge}. {III}.
\newblock {\em Publ. Math., Inst. Hautes {\'E}tud. Sci.}, 44:5--77, 1974.

\bibitem[Del80]{WeilII}
Pierre Deligne.
\newblock La conjecture de {Weil}. {II}.
\newblock {\em Publ. Math., Inst. Hautes {\'E}tud. Sci.}, 52:137--252, 1980.

\bibitem[Fu25]{furealreal}
Lie Fu.
\newblock Maximal real varieties from moduli constructions.
\newblock {\em Moduli}, 2:41, 2025.
\newblock Id/No e8.

\bibitem[GNA02]{Guillen-NA}
Francisco Guill{\'e}n and Vicente Navarro~Aznar.
\newblock An extension criterion for functors defined on smooth schemes.
\newblock {\em Publ. Math., Inst. Hautes {\'E}tud. Sci.}, 95:1--91, 2002.

\bibitem[HM25]{hinich}
Vladimir Hinich and Ieke Moerdijk.
\newblock {A}ddendum: localization of lax symmetric monoidal categories, 2025.
\newblock Available at \url{https://arxiv.org/abs/2509.03450}.

\bibitem[Jac17]{jacobson}
Jeremy Jacobson.
\newblock Real cohomology and the powers of the fundamental ideal in the {Witt}
  ring.
\newblock {\em Ann. \(K\)-Theory}, 2(3):357--385, 2017.

\bibitem[JX25]{jinxie}
Fangzhou Jin and Heng Xie.
\newblock On the real cycle class map for singular varieties.
\newblock {\em J. Homotopy Relat. Struct.}, 20(2):293--321, 2025.

\bibitem[Kel13]{thesekelly}
Shane Kelly.
\newblock Triangulated categories of motives in positive characteristic.
\newblock Preprint, Available at \url{https://arxiv.org/abs/1305.5349}, 2013.

\bibitem[KMS07]{milnorchowrevisited}
Moritz Kerz and Stefan M{\"u}ller-Stach.
\newblock The {Milnor}-{Chow} homomorphism revisited.
\newblock {\em \(K\)-Theory}, 38(1):49--58, 2007.

\bibitem[Lur17]{ha}
Jacob Lurie.
\newblock Higher algebra, 2017.
\newblock Available at \url{https://www.math.ias.edu/~lurie/papers/HA.pdf}.

\bibitem[MP11]{mccrory1}
Clint McCrory and Adam Parusi{\'n}ski.
\newblock The weight filtration for real algebraic varieties.
\newblock In {\em Topology of stratified spaces. Based on lectures given at the
  workshop, Berkeley, CA, USA, September 8--12, 2008}, pages 121--160.
  Cambridge: Cambridge University Press, 2011.

\bibitem[MP14]{mccrory2}
Clint McCrory and Adam Parusi{\'n}ski.
\newblock The weight filtration for real algebraic varieties. {II}: {Classical}
  homology.
\newblock {\em Rev. R. Acad. Cienc. Exactas F{\'{\i}}s. Nat., Ser. A Mat.,
  RACSAM}, 108(1):63--94, 2014.

\bibitem[OVV07]{ovv}
Dmitry Orlov, Alexander Vishik, and Vladimir Voevodsky.
\newblock An exact sequence for {{\(K^M_*/2\)}} with applications to quadratic
  forms.
\newblock {\em Ann. Math. (2)}, 165(1):1--13, 2007.

\bibitem[Pst23]{Pstragowski}
Piotr Pstr{\k{a}}gowski.
\newblock Synthetic spectra and the cellular motivic category.
\newblock {\em Invent. Math.}, 232(2):553--681, 2023.

\bibitem[Rob12]{rob2}
Marco Robalo.
\newblock Noncommutative {Motives} {I}: {A} {Universal} {Characterization} of
  the {Motivic} {Stable} {Homotopy} {Theory} of {Schemes}.
\newblock Preprint, available at \url{https://arxiv.org/abs/1206.3645}, 2012.

\bibitem[Ros96]{rost}
Markus Rost.
\newblock Chow groups with coefficients.
\newblock {\em Doc. Math.}, 1:319--393, 1996.

\bibitem[Sos19]{sosnilonegativeKtheory}
Vladimir Sosnilo.
\newblock Theorem of the heart in negative {{\(K\)}}-theory for weight
  structures.
\newblock {\em Doc. Math.}, 24:2137--2158, 2019.

\bibitem[Tot02]{totaro}
B.~Totaro.
\newblock Topology of singular algebraic varieties.
\newblock In {\em Proceedings of the international congress of mathematicians,
  ICM 2002, Beijing, China, August 20--28, 2002. Vol. II: Invited lectures},
  pages 533--541. Beijing: Higher Education Press; Singapore: World
  Scientific/distributor, 2002.

\bibitem[VSF11]{orange}
Vladimir Voevodsky, Andrei Suslin, and Eric~M. Friedlander.
\newblock {\em Cycles, Transfers, and Motivic Homology Theories.(AM-143),
  Volume 143}.
\newblock Princeton university press, 2011.

\end{thebibliography}

\end{document}